\documentclass[11pt]{amsart}

\usepackage{amsmath,amssymb,amsthm}

\newtheorem{theorem}{Theorem}[section]
\newtheorem{lemma}[theorem]{Lemma}

\newtheorem{proposition}[theorem]{Proposition}
\newtheorem{corollary}[theorem]{Corollary}
\newtheorem{assumption}[theorem]{Assumption}
\newtheorem{remark}[theorem]{Remark}

\newcommand{\A}{\mathcal{A}}
\newcommand{\D}{\mathcal{D}}

\newcommand{\cH}{\mathcal{H}}

\newcommand{\R}{\mathcal{R}}
\newcommand{\Tn}{T_n}
\newcommand{\Tnn}{T_{n+1}}

\newcommand{\Ahalf}{\mathcal{A}^{1/2}} 

\begin{document}


\title{Moore-Gibson-Thompson equation with memory, part II: general decay of energy}



\author{Irena Lasiecka and Xiaojun Wang*}


\begin{abstract}
We study a temporally third order (Moore-Gibson-Thompson) equation with a memory term. Previously it is known that, in non-critical regime, the global solutions exist and the energy functionals decay to zero. More precisely, it is known that the energy has exponential decay if the memory kernel decays exponentially. The current work is a generalization of the previous one (Part I) in that it allows the memory kernel to be more general and shows that the energy decays the same way as the memory kernel does, exponentially or not.
\end{abstract}



\maketitle

\section{Introduction}
We study the energy decay of Moore-Gibson-Thompson(MGT) equation with a viscoelastic term
\begin{equation}\label{mainmgt}
\tau u_{ttt}+ \alpha u_{tt}+c^2\A u+b\A u_t -\int_0^tg(t-s)\A u(s)ds=0,
\end{equation}
with initial data
\begin{equation}
u(0)=u_0, u_t(0)=u_1, u_{tt}(0)=u_2,
\end{equation}
where $\tau, c, b$ are parameters inherited from modeling process, see \cite{KLM11} and references therein. The constant $\alpha$ can be scaled out; we keep it, however, for notational consistency with \cite{KLM11}. $\A$ is a positive self-adjoint operator defined in a real Hilbert space $H$. The convolution term $\int_0^tg(t-s)\A u(s)ds$ reflects the memory effect of viscoelastic materials; the ``memory kernel" $g(t):[0,\infty)\rightarrow [0,\infty)$ directly relates to whether or how the energy decays.  Without this memory term, it is known the MGT equation has exponential energy decay in the non-critical regime, where $\gamma=\alpha-{c^2\tau\over b}>0$, see \cite{KLM11}. 

In an earlier work \cite{LW15a}, we studied (\ref{mainmgt}) with a nontrivial $g(t)$, but focusing on the case where $g(t)$ has exponential decay. We were able to get exponential decay of the energy for three types of memories in the non-critical regime. Here we study the case where the memory kernel has a more general decay rate. For sake of clarity, in this work we restrict our attention to one of the three types of memories introduced in \cite{LW15a}. 

{\it Notations:}
\begin{itemize}
\item $g(t)$: memory kernel.
\item $G(t)=\int_0^tg(s)ds$: strength of memory.
\item $H$: real Hilbert space.
\item $||\cdot||$: norms on $H$.
\item $g\circ h\triangleq \int_0^tg(t-s)||h(t)-h(s)||^2ds$, $g(\cdot)\in C(R^+), h(\cdot)\in H$.
\end{itemize}

\subsection{Main results}
 
Our work shows that the energy decay rate of system (\ref{mainmgt}), where memory effects get involved, is determined solely by the memory kernel: if $g(t)$ decays exponentially, then the energy decays exponentially too; if $g(t)$ decays slower, then the energy decays slower as well.
 
\begin{assumption}\label{main_assumption}
Let $G(t)=\int_0^t g(s)ds$. We assume
\begin{enumerate}
\item $g(t)\in C^1(\mathbb{R}_+), g(t)>0, g(0)<{b\alpha\gamma\over \tau^2}$ and $G(+\infty)<c^2$.
\item
There exists a convex function $H(\cdot) \in C^1(\mathbb{R}_+)$, which is strictly increasing with $H(0) = 0$, such that
$$g'(t)+H(g(t)) \leq 0, \forall t > 0.$$
\item
Let  $y(t) $ be  a solution of the following ODE 
$$y' (t) + H(y(t)) =0, y(0)=g(0),$$ and there exists $\alpha_0 \in (0,1)$ such that 
$y^{1-\alpha_0}(\cdot)  \in L_1(\mathbb{R}_+)$.
\item 
There exists $\bar{\delta}>0$
such that 
 $H(\cdot ) \in C^2(0,\bar{\delta})$ and 
$x^2H''(x)-xH'(x)+H(x)\geq 0, \forall x\in [0,\bar{\delta}]$.
\item $\A$ satisfies $||u||\leq \lambda_0||\Ahalf u||$ for all $u\in H$.
\item $\gamma=\alpha-{c^2\tau\over b}>0$.
\end{enumerate}
\end{assumption}

\begin{remark}
Comparing to the result in \cite{LW15a}, here we do not require the convexity of $g$, however $g(0)$ has to be suitably small. 
\end{remark}

\begin{theorem}[Existence of weak solution]\label{main_theorem0}
Consider system (\ref{mainmgt}).
If $$(u_0, u_1, u_2)\in \D(\Ahalf)\times\D(\Ahalf)\times H, $$
then, under the Assumption \ref{main_assumption}, this system has a unique weak solution $u$ satisfying
$$u\in C^1(\mathbb{R}_{+}; D( A^{\frac 12})) \cap C^{2}(\mathbb{R}_{+};H). $$
\end{theorem}

\begin{theorem}[General decay of energy]\label{main_theorem}
Under the Assumption \ref{main_assumption}, the energy of the weak solution decays to zero. Moreover, there exists positive constants $\tilde{C}, \tilde{\beta}, \tilde{\kappa}$ such that $$E(t)\leq \tilde{C}y(\tilde{\beta} t+\tilde{\kappa}).$$
\end{theorem}

\subsection{Background}

Moore-Gibson-Thompson(MGT) arises from modeling high amplitude sound waves.
There have been quite a few works in this research field due to the wide range of applications such as the medical and industrial use of high intensity ultrasound in lithotripsy, thermotherapy, ultrasound cleaning, etc. The classical nonlinear acoustics models include Kuznetsov's equation, the Westervelt equation and the Kokhlov-Zabolotskaya-Kuznetsov equation. A thorough study of the linearized models is a good starting point for better understanding the well-posedness and asymptotic behaviors of the nonlinear models. Actually, the work \cite{KLM11} has shown, even in the linear case, rich dynamics appear. In \cite{MMT12}, Marchand et.al. 2012 gave a detail analysis of this equation; using the abstract semigroup approach and a refined spectrum analysis they settle the well-posedness of (\ref{mainmgt0}) and identified an accumulation point of eigenvalues which essentially connects to the exponential decay of energy. Kaltenbacher et.al. \cite{KLP12} also studied the fully nonlinear version of MGT equation and gave the well-posedness and the exponential decay.

In \cite{KLM11}, Kaltenbacher, Lasiecka and Marchand studied well-posedness and uniform decays of energy for  the linearized MGT system, 
\begin{eqnarray}\label{mainmgt0}
\tau u_{ttt}+ \alpha u_{tt}+c^2\A u+b\A u_t=0.
\end{eqnarray} A critical parameter $\gamma=\alpha-{c^2\tau\over b}$ was disclosed. It has been shown in \cite{KLM11} that when $\gamma>0$, namely in the non-critical case, the problem is well-posed and its solution is exponentially stable; while $\gamma=0$, the energy is conserved.

It is well known that the {\it wave equation} conserves mechanical energy. More specifically, consider the wave equation
$$u_{tt}-\Delta u=0.$$ Multiply by $u_t$, we have
$${d\over dt}{1\over 2}(||u_t||^2+||\nabla u||^2)=0\Rightarrow E(t)\equiv{1\over 2}(||u_t||^2+||\nabla u||^2)=E(0).$$
Here we use $E(t)$ to represent the mechanical energy, the summation of kinetic and potential energy. So the mechanical energy does not change when a wave evolves. On the other hand, we normally see mechanical energy dissipates in a physical system due to certain damping mechanism.\footnote{Damping is the dissipation of energy, which transform the mechanical energy into another form, e.g. heat or light.} A different viewpoint is, in order to force the energy decay, we have to inject damping mechanisms into the system. In this aspect, there are different ways to implement the idea. One way is to add viscous damping, also called frictional damping, into the wave equation. Namely, taking into account the friction, we end up with equation in form of
$$u_{tt}-\Delta u+u_t=0.$$ 
The energy estimate gives $${d\over dt}{1\over 2}(||u_t||^2+||\nabla u||^2)=-||u_t||^2\leq 0\Rightarrow E(t)\leq E(0).$$
Obviously, the mechanical energy does not increase for sure; in fact it decays, at least when the kinetic $||u_t||^2$ is not zero. Indeed, it can be shown, by standard Lyapunov function method, that the energy decays exponentially \cite{Che79}.
Adding the structure damping $\Delta u_t$ to the wave equation is another way to obtain exponential decay of the energy. These topics are of great importance in real applications and form an active research area \cite{Adh00}.

What interests us here is a different damping mechanism caused by viscoelasticity, which forces the appearance of memory term in a system. Viscoelasticity is  the property of materials that exhibit both viscous and elastic characteristics when undergoing deformation. It usually appears in fluids with complex microstructure, such as polymers, suspensions, and granular materials. One encounters viscoelastic materials in biological science, materials sciences as well as in many industrial processes, e.g., in the chemical, food, and oil industries. The phenomena and mathematical models for such materials are more varied and complex than those of pure elastic materials or those of pure Newtonian fluids, see\cite{Ren00}.

In \cite{LW15a}, we investigate the case where memory effects are incorporated into the model. 
It is well known that memory generates stabilizing mechanism for second order wave equations. There is an abundant literature on the topic, see \cite{ACS08, AC09, Ala10, CCM08, CCLW15, FP02, HW10, LMM13, Mes08, MM12, RS01, XL13}, to name a few. The prototype  model is viscoelastic  wave equation,
\begin{equation}\label{wave_memory}
u_{tt}-\Delta u+\int_0^tg(t-s)\Delta u(s)ds=0,
\end{equation}
where $\Delta$ is Laplacian defined on bounded domain with smooth boundary.

The convolution term $\int_0^tg(t-s)\A u(s)ds$ represents memory: the integral itself suggests the nonlocality in time; the system at present moment ``remember" certain information in the past history. This is a simplified one dimensional wave equation with memory. Like all the physically meaningful partial differential equations, it is derived based on the balance laws of momentum and/or conservation of mass, together with a constitutive equation relating stress to strain (elasticity) or stress to rate of strain (fluids). For more information on the derivation of this 1-D system, we refer to Renardy \cite{HNR87}. On a detailed description on damping mechanisms, see \cite{Adh00} and the references therein. 

In the study of wave equation, an interesting phenomena is that, when two different kind of damping terms appear simultaneously, despite the fact that both terms help the system dissipate the energy, they do not necessarily accelerate the decay process. For instance, while the frictional damping alone causes exponential decay; only polynomial decay can be reached when an additional memory term presents itself with a polynomial kernel\cite{CO03}. 

Natural questions can be raised based on the study of wave-memory system: What kind of term creates damping effect in a system?  How to select multipliers that can ``detect" the damping? Does the damping term always help? Is it possible that adding a damping term ``hurt" the initial energy, in any sense? While the answers to these questions in case of wave with frictional damping are relatively simple, they are not for the memory damping we introduce below. Following the wave-memory system, this work is part of the effort in understanding this damping mechanism. To our best knowledge, this is one of the first few works that study the memory damping in the context of a {\it third order} in time system. 

In the rest of this section, we recall the results for MGT equation. In Section 2 we prove the main results stated before.

\subsection{Results for MGT}

Consider the MGT equation (\ref{mainmgt0}),
$$\tau u_{ttt}+ \alpha u_{tt}+c^2\A u+b\A u_t=0,$$ with initial conditions
$$u(0)=u_0, u_t(0)=u_1, u_{tt}(0)=u_2.$$
Here $\A$ is a self-adjoint positive operator on a Hilbert space $H$ with a dense domain $\D(\A)\subset H$. The initial data $$(u_0, u_1, u_2)\in \cH\equiv\D(\Ahalf)\times\D(\Ahalf)\times H.$$

We state the following results from \cite{KLM11}:

\begin{theorem}[Kaltenbacher et.al. 2011, \cite{KLM11}]\label{KLM11_1}

Let $\tau>0, b>0, \alpha\in \R$, then the system (\ref{mainmgt0}) generates a strongly continuous semigroups on $\cH$. 
\end{theorem}
Moreover, introduce the parameter $\gamma=\alpha-{c^2\tau\over b}$, we have

\begin{theorem}[Kaltenbacher et.al. 2011, \cite{KLM11}]\label{KLM11_2}

Let $\tau>0, b>0, \alpha>0, c>0$. Let $$\hat{E}_1(t)=b||\Ahalf(u_t+{c^2\over b}u)||^2+\tau||u_{tt}+{c^2\over b}u_t||^2+{c^2\over b}\gamma ||u_t||^2,$$
$$\hat{E}_2(t)=\alpha||u_t||^2+c^2||\Ahalf u||^2,$$
and $$\hat{E}(t)=\hat{E}_1(t)+\hat{E}_2(t).$$

1. If $\gamma>0$, then the semigroup generated in Theorem \ref{KLM11_1} is exponentially stable on $\cH$. And there exist $\omega>0, C>0$ such that $$\hat{E}(t)\leq Ce^{-\omega t}\hat{E}(0), t>0.$$

2. If $\gamma=0$, the energy $\hat{E}_1(t)$ remains constant.
\end{theorem}

\begin{remark}\label{substitution}
The MGT equation (\ref{mainmgt0}) can be rewritten as $$\tau z_{tt}+b\A z+\gamma z_t={\gamma c^2\over b}u_t,$$ where $z(t)=u_t+{c^2\over b}u$. Since $\tau>0, b>0$, it is easy to see that, in the new variable $z$, (\ref{mainmgt0}) becomes wave equation when $\gamma=0$, hence no decay can be expected. This is exactly how MGT is connected to wave equation through the parameter $\gamma=\alpha-{c^2\tau\over b}$. On the other hand, when $\gamma>0$,  (\ref{mainmgt0}) becomes a damped wave equation which intends to force exponential decay.
\end{remark}

\begin{remark}
Theorems \ref{KLM11_1} and \ref{KLM11_2} were firsted proved in \cite{KLM11}. A different proof of Theorem \ref{KLM11_2}, which can serve a warm up for the current work, was given in \cite{LW15a}.
\end{remark}

\begin{remark}
In this work, we work in the non-critical regime ($\gamma>0$). Both the ``frictional damping" and ``memory damping" are present in the system. It turns out that, while the frictional damping alone causes exponential decay, we are only able to get a slower decay if a weaker memory damping is added to the system.
\end{remark}

From now on, we study MGT with memory terms and prove the main results stated before.
\section{Proofs}
\begin{remark}[On the existence, Theorem \ref{main_theorem0}]
We omit the proof and focus on the decay part. The concept of weak solutions and the corresponding functional setting are from \cite{KLM11}. 
The proof can be completed by the standard Galerkin method. The key is a global energy bound, which is a by product of the decay part proof below. \\
\end{remark}

{\bf Proof of the general decay, Theorem \ref{main_theorem}}\\
The proof is lengthy and quite involved. So we split the proof into three subsections. 
In {\it the first subsection}, we derive the energy inequality. All the later calculations are based on this estimate. However, we are not able to get decay rate out of this continuous version of energy estimate, for it is not in an form applicable to Gronwall inequality. To get what we want, we need two parts: 1. the discretized version of energy estimate developed in \cite{LT93}, which essentially is a generalized version of Gronwall 
inequality; 2. the iteration technique developed in \cite{LW14}, which can optimize the estimate in finite steps.
In {\it the second subsection}, we discretize the estimate,  based on which we are able to get an ``initial" decay estimate. The decay rate here is not optimal. 
We improve the result in {\it the third subsection} through an iteration process. The iteration can be finished in finite steps and we are able to get the optimal decay rate.

\subsection{Energy inequality}
In this subsection we carry out the energy estimate in several steps. The goal is to get a differential inequality in form of ${d\over dt}E(t)=-R(t)$, where $E(t)$ is the {\it natural energy functional} and $R(t)$ is a positive function called {\it damper}: it forces $E(t)$ to decrease and creates energy damping. Generally, the expression of natural energy $E(t)$ looks sloppy, hence {\it standart energy functional} $F(t)$ is defined and proved to be equivalent to $E(t)$. Namely there are positive constants $C_1, C_2$ such that $C_1E(t)\leq F(t)\leq C_2E(t)$. As a dressed up version of $E(t)$, $F(t)$ usually makes the calculations much neater.\\

Recall the equation
\begin{equation}\label{mainmgt1}
\tau u_{ttt}+ \alpha u_{tt}+c^2\A u+b\A u_t -\int_0^tg(t-s)\A u(s)ds=0.
\end{equation}

{\bf Step} I. Multiplying (\ref{mainmgt1}) by $u_{tt}$, we have

$${d\over dt}E_1(t)=-2\alpha||u_{tt}||^2+2c^2||\Ahalf u_t||^2-2g(t)(\A u,u_t)$$
$$+2\int_0^tg'(t-s)(\A (u(t)-u(s)),u_t(t))ds, $$ with
$$E_1(t)=[\tau ||u_{tt}||^2+b||\A ^{1/2}u_t||^2+2c^2(\A u, u_t)]$$
$$-2\int_0^tg(t-s)(\A u(s),u_t(t))ds.$$

{\bf Step} II. Multiplying (\ref{mainmgt1}) by $u_t$, we have
$${d\over dt}E_2(t)=-2b||\Ahalf u_t||^2+2\tau||u_{tt}||^2+g'\circ\Ahalf u-g(t)||\Ahalf u||^2,$$ with
$$E_2(t)=[c^2 ||\Ahalf u||^2+\alpha||u_t||^2+2\tau (u_{tt},u_t)]$$
$$+g\circ\Ahalf u-\int_0^tg(s)ds||\Ahalf u||^2.$$
Recall $g\circ\Ahalf u=\int_0^tg(t-s)||\Ahalf u(t)-\Ahalf u(s)||^2ds$. \\

{\bf Step} III. {\it Define natural energy functional}. Since $\gamma=\alpha-{c^2\tau\over b}>0$ is assumed, we can pick a $k$ such that $${c^2\over b}<k<{\alpha\over\tau}.$$
Let $E(t)=E_1(t)+kE_2(t)$ be the {\it natural energy}, we have
\begin{eqnarray}\label{natural_energy}
E(t)&=&b||\Ahalf u_t||^2+2c^2(\A u, u_t)+c^2k||\Ahalf u||^2\nonumber\\
&+&\tau||u_{tt}+ku_t||^2+k\tau({\alpha\over \tau}-k)||u_t||^2\nonumber\\
&+&kg\circ\Ahalf u-kG(t)||\Ahalf u||^2\\
&+&2\int_0^tg(t-s)(\Ahalf (u(t)-u(s)), \Ahalf u_t(t))ds\nonumber\\
&-&2G(t)(\A u(t), u_t(t)).\nonumber
\end{eqnarray}
 and

\begin{eqnarray}\label{energy_equality}
&&{d\over dt}E(t)+R(t)=0, ~with~\nonumber\\
R(t)&&=2(\alpha-k\tau)||u_{tt}||^2+2(bk-c^2)||\Ahalf u_t||^2\nonumber\\
&&+2g(\A u,u_t)+kg||\Ahalf u||^2\nonumber\\
&&-2\int_0^tg'(t-s)(\A (u(t)-u(s)),u_t(t)ds-kg'\circ\Ahalf u.
\end{eqnarray}

\begin{remark}
The calculations appear to be lengthy. But two hints should make it less tedious: 1. since we are working on real Hilbert space, many calculations here are like algebras of completing square; 2. both $E(t)$ and $R(t)$ contain two parts, one part contains terms from the normal MGT equation, the other stands for memory terms.
\end{remark}
\begin{remark}
Although the calculations here are for arbitrary $k\in ({c^2\over b}, {\alpha\over\tau})$, we will need $k$ to satisfy an additional condition stated later, see Lemma \ref{kg_lemma}.
\end{remark}

{\bf Step} IV. {\it Define standard energy functional}. Now denote the {\it standard energy}$$F(t)=||u_{tt}||^2+||\Ahalf u_t||^2+||\Ahalf u||^2+g\circ\Ahalf u,$$
we claim that 
\begin{lemma}
$E(t)\sim F(t)$.
\end{lemma}

\begin{remark}
Here $F(t)$ contains a memory term in form of $g\circ\Ahalf u$, contrary to the energy functional defined in \cite{LW151}, which contains memory term $-g'\circ\Ahalf u$ instead. This is the consequence of us trying to eliminate the convexity restriction on $g(t)$.
\end{remark}
\begin{proof}
We want to prove that one is bounded by a multiple of the other.
\begin{eqnarray*}
E(t)=&&b||\Ahalf u_t||^2+2c^2(\A u, u_t)+c^2k||\Ahalf u||^2\\
&&+\tau||u_{tt}+ku_t||^2+k\tau({\alpha\over \tau}-k)||u_t||^2\\
&&+kg\circ\Ahalf u-kG(t)||\Ahalf u||^2\\
&&+2\int_0^tg(t-s)(\Ahalf u(t)- \Ahalf u(s), \Ahalf u_t(t))ds\\
&&-2G(t) (\A u(t), u_t(t))
\end{eqnarray*}
\begin{eqnarray*}
=&&{c^2-G(t)\over k}||\Ahalf u_t+k\Ahalf u||^2+(b-{c^2\over k})||\Ahalf u_t||^2\\
&&+\tau||u_{tt}+ku_t||^2+k\tau({\alpha\over \tau}-k)||u_t||^2\\
&&+{G(t)\over k}||\Ahalf u_t||^2+kg\circ\Ahalf u\\
&&+2\int_0^tg(t-s)(\Ahalf u(t)- \Ahalf u(s), \Ahalf u_t(t))ds.
\end{eqnarray*}

\begin{remark}
The summation of last three terms is non-negative, since
$$2|\int_0^tg(t-s)(\Ahalf u(t)- \Ahalf u(s), \Ahalf u_t(t))ds|$$
$$\leq 2\int_0^tg(t-s)\sqrt{k}||\Ahalf u(t)- \Ahalf u(s)|| {1\over \sqrt{k}}||\Ahalf u_t(t)||ds$$
$$\leq \int_0^tg(t-s)[k||\Ahalf u(t)- \Ahalf u(s)||^2+{1\over k}||\Ahalf u_t(t)||^2]ds$$
$$=kg\circ\Ahalf u+{G(t)\over k}||\Ahalf u_t||^2, ~with~G(t)=\int_0^tg(s)ds.$$
\end{remark}

\begin{remark}[Energy match]\label{energy_match}
While we can guarantee the non-negativeness of the summation, that is not sufficient for proving the equivalence of $E(t)$ and $F(t)$. For that purpose, we expect a positive copy of $g\circ\Ahalf u$ to come out. This can be done by {\it energy match}. The idea is to borrow a portion of $||\Ahalf u_t||^2$ from term $(b-{c^2\over k})||\Ahalf u_t||^2$ and add it to term ${G(t)\over k}||\Ahalf u_t||^2$. Then through completing square, we obtain a substantial portional of $g\circ\Ahalf u$.
\end{remark}

Split the term $(b-{c^2\over k})||\Ahalf u_t||^2$ to get
\begin{eqnarray*}
E(t)=&&{c^2-G(t)\over k}||\Ahalf u_t+k\Ahalf u||^2+{1\over 2}(b-{c^2\over k})||\Ahalf u_t||^2\\
&&+\tau||u_{tt}+ku_t||^2+k\tau({\alpha\over \tau}-k)||u_t||^2\\
&&+[{1\over 2}(b-{c^2\over k})+{G(t)\over k}]||\Ahalf u_t||^2+kg\circ\Ahalf u\\
&&+2\int_0^tg(t-s)(\Ahalf u(t)- \Ahalf u(s), \Ahalf u_t(t))ds.
\end{eqnarray*}

Pick a $\sigma_0$, so that $0<\sigma_0<k$ and ${G(t)\over k-\sigma_0}<[{1\over 2}(b-{c^2\over k})+{G(t)\over k}]$. This can be done since $G(t)$ is finite and we can choose $\sigma_0$ sufficiently small. Then by Holder inequality we have
 $$2|\int_0^tg(t-s)(\Ahalf u(t)- \Ahalf u(s), \Ahalf u_t(t))ds|\leq (k-\sigma_0)g\circ\Ahalf u+{G(t)\over (k-\sigma_0)}||\Ahalf u_t||^2.$$ So 
\begin{eqnarray*}
E(t)\geq&&{c^2-G(t)\over k}||\Ahalf u_t+k\Ahalf u||^2+{1\over 2}(b-{c^2\over k})||\Ahalf u_t||^2\\
&&+\tau||u_{tt}+ku_t||^2+k\tau({\alpha\over \tau}-k)||u_t||^2\\
&&+[{1\over 2}(b-{c^2\over k})+{G(t)\over k}]||\Ahalf u_t||^2+kg\circ\Ahalf u\\
&&-(k-\sigma_0)g\circ\Ahalf u-{G(t)\over (k-\sigma_0)}||\Ahalf u_t||^2\\
\geq&&{c^2-G(t)\over k}||\Ahalf u_t+k\Ahalf u||^2+{1\over 2}(b-{c^2\over k})||\Ahalf u_t||^2\\
&&+\tau||u_{tt}+ku_t||^2+k\tau({\alpha\over \tau}-k)||u_t||^2\\
&&+\sigma_0 g\circ\Ahalf u.
\end{eqnarray*}
Since $c^2-G(t)$ is strictly positive, we have $E(t)\geq CF(t)$ by Lemma \ref{square_lemma} below.

\begin{lemma}\label{square_lemma}
Let $C_0>0$. Then 
$$||f+g||^2+C_0||g||^2\sim ||f||^2+||g||^2.$$
\end{lemma}

\begin{proof}
Note $$||f+g||^2+C_0||g||^2={1\over {1+{C_0\over 2}}}||f||^2+2(f,g)+(1+{C_0\over 2})||g||^2$$
$$+(1-{1\over {1+{C_0\over 2}}})||f||^2+{C_0\over 2}||g||^2$$
$$\geq C_1(||f||^2+||g||^2), \mbox{  } C_1=\min\{(1-{1\over {1+{C_0\over 2}}}),{C_0\over 2}\}.$$ 
The other direction $||f+g||^2+C_0||g||^2\leq C_2(||f||^2+||g||^2)$ is trivial.
\end{proof}

On the other hand, it is easy to see
\begin{eqnarray*}
E(t)=&&{c^2-G(t)\over k}||\Ahalf u_t+k\Ahalf u||^2+(b-{c^2\over k})||\Ahalf u_t||^2\\
&&+\tau||u_{tt}+ku_t||^2+k\tau({\alpha\over \tau}-k)||u_t||^2\\
&&+{G(t)\over k}||\Ahalf u_t||^2+kg\circ\Ahalf u\\
&&+2\int_0^tg(t-s)(\Ahalf u(t)- \Ahalf u(s), \Ahalf u_t(t))ds\\
\leq&&C||\Ahalf u_t+k\Ahalf u||^2+C||\Ahalf u_t||^2\\
&&+C||u_{tt}+ku_t||^2+C||u_t||^2\\
&&+C||\Ahalf u_t||^2+Cg\circ\Ahalf u\\
\leq&&C||u_{tt}||^2+C||u_t||^2+C||\Ahalf u_t||^2\\
&&+C||\Ahalf u||^2+Cg\circ\Ahalf u\\
\leq && CF(t)
\end{eqnarray*} because of the boundedness of $g(t), G(t)$. The equivalence of $E(t)$ and $F(t)$ is established.
\end{proof}
\begin{remark}
Obviously $F(t)$ is much simpler than $E(t)$. Indeed, we have cleaned up the cross terms and the coefficients. Moreover, since $||w||\leq \lambda_0||\Ahalf w||$ for all $w\in H$, $F(t)$ also absorbs terms $||u_t||^2$ and $||u||^2$.
\end{remark}

{\bf Step} V. {\it Analyze the damper}. Recall 
$${d\over dt}E(t)=-R(t)$$ with
$$R(t)=2(\alpha-k\tau)||u_{tt}||^2+2(bk-c^2)||\Ahalf u_t||^2+2g(\A u,u_t)$$
$$+kg||\Ahalf u||^2-2\int_0^tg'(t-s)(\A (u(t)-u(s)),u_t(t)ds-kg'\circ\Ahalf u.$$

We claim $R(t)\geq 0$ and hence $E(t)$ is non-increasing.
$$R(t)=2(\alpha-k\tau)||u_{tt}||^2+2(bk-c^2)||\Ahalf u_t||^2$$
$$+{g\over k}||\Ahalf u_t+k\Ahalf u||^2-{g\over k}||\Ahalf u_t||^2-kg'\circ\Ahalf u$$
$$-2\int_0^tg'(t-s)(\A (u(t)-u(s)),u_t(t)ds.$$

For the last term, we can pick a number $\delta\in (0,k)$ such that
$$-2\int_0^tg'(t-s)(\A (u(t)-u(s)),u_t(t)ds$$
$$=2\int_0^t\sqrt{-g'(t-s)}\sqrt{k-\delta}\sqrt{-g'(t-s)}\sqrt{1/k-\delta}(\A (u(t)-u(s)),u_t(t)ds.$$
$$\leq -(k-\delta)g'\circ\Ahalf u-{1\over k-\delta}||\Ahalf u_t||^2\int_0^tg'(t-s)ds$$
$$= -(k-\delta)g'\circ\Ahalf u+{g(0)-g(t)\over k-\delta}||\Ahalf u_t||^2$$

So 
$$R(t)\geq 2(\alpha-k\tau)||u_{tt}||^2+{g(t)\over k}||\Ahalf u_t+k\Ahalf u||^2$$
$$+2(bk-c^2-{g(0)\over k-\delta})||\Ahalf u_t||^2$$
$$-\delta g'\circ\Ahalf u+({1\over k-\delta}-{1\over k})g(t)||\Ahalf u_t||^2.$$

Now, given the Assumption \ref{main_assumption}, by the Lemma \ref{kg_lemma} below, we can pick $\sigma>0$ such that $2(bk-c^2-{g(0)\over k-\delta})$ becomes strictly positive. 

\begin{lemma}\label{kg_lemma}
If $0<g(0)<{b\alpha\gamma\over \tau^2}$, then there are $\sigma>0$ and $k\in ({c^2\over b}, {\alpha\over\tau})$ such that $g(0)<(k-\sigma)(bk-c^2)$, or equivalently, $2(bk-c^2-{g(0)\over k-\sigma})>0$.
\end{lemma}

\begin{proof}
Since $g(0)<{b\alpha\gamma\over \tau^2}={\alpha\over\tau}(b{\alpha\over\tau}-c^2)$, 
all we need is to show $$(k-\sigma)(bk-c^2)\rightarrow {\alpha\over\tau}(b{\alpha\over\tau}-c^2) ~as~  k\rightarrow {\alpha\over\tau}, \sigma\rightarrow 0,$$ which is trivially true.
\end{proof}

Thus there exists a positive constant $C_\sigma$ such that
$$R(t) \geq C_\delta (||u_{tt}||^2+||\Ahalf u_t||^2-g'\circ\Ahalf u).$$
Obviously, 
\begin{equation}\label{c_sigma}
-g'\circ\Ahalf u\leq {1\over C_\sigma}R(t).
\end{equation}
This inequality will be used later, where we construct convex function and apply Jenson's inequality. See lemmas \ref{est2} and \ref{est3}.

Moreover, integrate (\ref{energy_equality}) on $[t, T]$, we have $E(T)+ \int_t^T R(s)ds= E(t)$. Hence
$$E(T)+ C_\delta \int_t^T ||u_{tt}||^2+||\Ahalf u_t||^2-g'\circ\Ahalf uds\leq E(t).$$
In particular
\begin{equation}\label{part1}
\int_t^T||u_{tt}||^2+||\Ahalf u_t||^2-g'\circ\Ahalf uds\leq CE(t).
\end{equation}

What we expect now is $$\int_t^T||\Ahalf u||^2ds\leq CE(t).$$
To achieve the goal, multiply (\ref{mainmgt}) by $u$, we have 
$${d\over dt}{b\over 2}||\Ahalf u||^2+c^2||\Ahalf u||^2$$
$$=\alpha||u_t||^2+{d\over dt}[{\tau\over 2}||u_t||^2-\tau(u_{tt},u)-\alpha(u_t,u)]$$
$$+G(t)||\Ahalf u||^2-\int_0^tg(t-s)(\Ahalf (u(t)-u(s)),\Ahalf u(t))ds.$$ 
which gives
$$(c^2-G(t)-\epsilon G(t))\int_t^T||\Ahalf u||^2ds$$
$$\leq\alpha\int_t^T||u_t||^2ds+[{\tau\over 2}||u_t||^2-\tau(u_{tt},u)-\alpha(u_t,u)]|_t^T$$
$$+{1\over 4\epsilon}\int_t^Tg\circ\Ahalf uds-{b\over 2}||\Ahalf u||^2|_t^T.$$ 

By Assumption \ref{main_assumption}-1, $G(t)\leq G(+\infty)<c^2$, we can choose $\epsilon$ so small that $$(c^2-G(t)-\epsilon G(t))\geq C_\epsilon>0.$$
Furthermore,
$$\int_t^T||u_t||^2ds\leq \lambda_0^2\int_t^T||\Ahalf u_t||^2ds\leq CE(t),$$
$$||u_t||^2|_t^T\leq ||\Ahalf u_t(t)||^2+||\Ahalf u_t(T)||^2\leq CE(t),$$ 
$$||(u_{tt},u)||\leq ||u_{tt}||^2+||u||^2\leq CE(t),$$ 
$$||(u_{t},u)||\leq CE(t),$$ 
$$\int_t^Tg\circ\Ahalf uds\leq CE(t),$$
hence
\begin{equation}\label{part2}
\int_t^T||\Ahalf u||^2ds\leq CE(t).
\end{equation}

Now combine (\ref{part1}) and (\ref{part2}), we have
$$\int_t^T||u_{tt}||^2+||\Ahalf u_t||^2+||\Ahalf u||^2-g'\circ\Ahalf uds\leq CE(t).$$

\begin{remark}
Consider the special case where $H(s)=cs$, hence $g'\leq -cg$, then we can get exponential decay without going further. Indeed, since $-g'\geq cg$, from the last inequality we have $\int_t^T E(s)ds\leq C\int_t^T||u_{tt}||^2+||\Ahalf u_t||^2+||\Ahalf u||^2-g'\circ\Ahalf uds\leq CE(t).$ By a generalized Gronwall lemma, $\int_t^T E(s)ds\leq CE(t)$ immediately gives the exponential decay. See Lemma 2.7 in \cite{LW15a}.
\end{remark}
\begin{remark}
In this work, we assume memory kernel $g(t)$ has general decay rate, and we need techniques that is more sophisticate.
To do the decay estimate, we adopt the idea in Lasiecka,etc \cite{LT93} and represent the decay by the solution of an ODE. For that reason we have to discretize the energy inequality and apply a series of lemmas. The technique has been applied in \cite{LW14}.
\end{remark}

{\bf Step} VI.

The previous step implies
$$\int_t^T||u_{tt}||^2+||\Ahalf u_t||^2+||\Ahalf u||^2\leq CE(t).$$
So
\begin{eqnarray*}
\int_t^TE(s)ds&&\leq C_1E(t)+C_2\int_t^Tg\circ\Ahalf uds\\
&&=C_1E(T)+C_1\int_t^TR(s)ds+C_2\int_t^Tg\circ\Ahalf uds.
\end{eqnarray*}

Since $E(t)$ is non-increasing function of $t$, we arrive at
\begin{equation}\label{sixth_inequality}
(T-t-C_1)E(T)\leq C_1\int_t^TR(s)ds+C_2\int_t^Tg\circ\Ahalf uds.
\end{equation}

\begin{remark}
We emphasize the fact that the starting ``moment" $t$ is taken sufficiently large so that the energy estimate can be carried out. Other than that, $t$ and $T$ are arbitrary. The constants $C_1, C_2$ are independent of $t$ and $T$.
As we shall see soon, the decay rate shall be discovered based on the relations between each pari of the following integrals $$\int_t^Tg\circ\Ahalf uds\leftrightarrow\int_t^T(-g)'\circ\Ahalf uds\leftrightarrow \int_t^TR(s)ds.$$ 
\end{remark}

\subsection{Initial decay estimate} 
We introduce the Lemma 3.3 from \cite{LT93}, which is the cornerstone of the later calculations.
\begin{lemma}[Lasiecka \& Tataru 1993,\cite{LT93}]\label{lemmaLT93} 
Let $p$ be a positive, increasing function such that $p(0)=0$. Since $p$ is increasing, we can define an increasing function $q, q(x)\equiv x-(I+p)^{-1}(x)$. Consider a sequence $s_n$ of positive numbers which satisfies 
$$s_{m+1}+p(s_{m+1})\leq s_m.$$ Then $s_m\leq S(m)$ where $S(t)$ is a solution of the differential equation 
$${d\over dt}S(t)+q(S(t))=0, S(0)=s_0.$$ Moreover, if $p(x)>0$ for $x>0$ then $\lim_{t\rightarrow\infty}S(t)=0$.
\end{lemma}

\begin{remark}
This lemma suggests that we are expecting an discretized energy inequality in form of $E(\Tnn)+\hat{H}(E(\Tnn))\leq E(\Tn)$, with an increasing non-negative function $\hat{H}$. 
\end{remark}
We put the calculations in several steps.

{\bf Step 1: Discretize energy estimate}\\

In (\ref{sixth_inequality}), pick $t=\Tn=nT_0$ and $T=\Tnn=(n+1)T_0$, where $T_0$ is a given positive real number and $n$ is any positive integer. Here $T_0$ is chosen to be sufficiently large so that all the estimates before can be justified.
We have
\begin{equation}
(T_0-C_1)E(\Tnn)\leq C_1\int_{\Tn}^{\Tnn}R(s)ds+C_2\int_{\Tn}^{\Tnn}g\circ\Ahalf uds.\nonumber
\end{equation}

Here we pick $T_0$ large enough so that $T_0-C_1>0$ and we have
\begin{eqnarray}\label{7th_ineq}
E(\Tnn)\leq C_1\int_{\Tn}^{\Tnn}R(s)ds+C_2\int_{\Tn}^{\Tnn}g\circ\Ahalf uds.
\end{eqnarray}
Note $n$ is arbitrary positive integer and $C_1, C_2$ depend on $T_0$, but do not depend on $n$.\\

{\bf Step 2: Construct convex function related to $H$}\\
For the $\alpha_0\in(0,1)$ given in Assumption \ref{main_assumption}, define function
\begin{equation}
H_{1,\alpha_0}(s) \triangleq  \alpha_0 s^{1- \frac{1}{\alpha_0}} H_{\alpha_0} (s)  = \alpha_0  s^{1- \frac{1}{\alpha_0}} H (s^{\frac{1}{\alpha_0}} ).
\end{equation}
 
\begin{lemma}\label{est2}
Denote $||\Ahalf u(t)-\Ahalf u(s)||$ by $f(t,s)$. Under the {\bf Assumptions \ref{main_assumption}}, there exists an interval $[0,\delta), 0<\delta<\bar{\delta}$ on which we have

1. ${H}_{1,\alpha_0}(0)=0$ and ${H}_{1,\alpha_0}(s)$ is increasing and convex. 

2. Moreover, if $c_{\alpha_0}\triangleq \sup\limits_{t>0}c(\alpha_0,t)=\sup\limits_{t>0}\int_0^tg^{1-\alpha_0}(t-s)f^2(t,s)ds<\infty$, then there exist constants $\vartheta$ so that
\begin{equation}\label{H_T}
{H}_{1,\alpha_0}[\vartheta (g \circ  A^{\frac 12} u)](t) 
\leq {\alpha_0\vartheta\over C_\sigma} R(t) , \mbox{ for}~t \in [\Tn, \Tnn ] , ~~ n =1,2 \ldots , 
\end{equation} with $\vartheta \in (0,1) $ independent of $n$. 
\end{lemma}

\begin{proof} The proof was given in \cite{LW14}. We present it here for readers' convenience.
Since $H\in C^1[0,\infty)$, it is easy to see that function ${H}_{1,\alpha_0}(s)=\alpha_0 s^{1-{1\over\alpha_0}}H(s^{1\over \alpha_0}), 0<\alpha_0<1$ is well-defined on $[0,\infty)$ with ${H}_{1,\alpha_0}(0)=0$.

1.
Let $k={1 \over\alpha_0}\geq 1$. We have ${H}_{1,\alpha_0(}s)={1\over k}s^{1-k}H(s^k)$ and
$${H}_{1,\alpha_0}'(s)={(1-k)\over k}s^{-k}H(s^k)+s^{1-k}H'(s^k)s^{k-1}$$
$$={1\over k}s^{-k}H(s^k)+[H'(s^k)-s^{-k}H(s^k)]>0, \forall s>0.$$
The last inequality results from the properties of $H(s)$
\begin{equation}\label{asmpB2}
H(x)\leq xH'(x), x>0.
\end{equation}
Indeed, this follows from  geometric interpretation of convexity of  $H(x) $ and the fact that $H(0) =0 $ which then gives $H'(x) x - H(x) > 0, \forall x>0$.
Thus we show that ${H}_{1,\alpha_0}(x)$ is increasing on the positive half line.

For the second derivative on $(0,\infty)$, we have
\begin{eqnarray}\label{num1}
{H}_{1,\alpha_0} ''(s)=&&(k-1)s^{-k-1}H(s^k)\nonumber\\
&&+(1-k)s^{-k}H'(s^k)s^{k-1}+ks^{k-1}H''(s^k)\nonumber\\
=&&{1\over s^{k+1}}[kx^2H''(x)-(k-1)xH'(x)+(k-1)H(x)]\nonumber\\
=&&{k-1\over s^{k+1}}\big(x^2H''(x)-xH'(x)+H(x)\big)
+{x^2\over s^{k+1}}H''(x),
\end{eqnarray} with $x=s^k$.

In view of Assumption \ref{main_assumption}, we conclude the proof for part 1, namely, there exist an interval $(0,\delta), 0<\delta<\bar{\delta}$ on which ${H}_{1,\alpha_0} (s)=\alpha_0 s^{1-{1\over\alpha_0}}H(s^{1\over\alpha_0})$ is increasing and convex.

{\bf 2.}

Now we prove that  ${H}_{1,\alpha_0}(s)=\alpha_0 s^{1-{1\over\alpha_0}}H(s^{1\over \alpha_0})$ satisfies
\begin{equation}\label{iteration2}
{H}_{1,\alpha_0}[\vartheta(g\circ A^{1\over 2}u)(t)]\leq {\alpha_0\vartheta\over C_\sigma} R(t)
\end{equation}
(or equivalently $(g\circ A^{1\over 2}u)(t)\leq \frac 1\vartheta {H}_{1,\alpha_0}^{-1}({\alpha_0\vartheta\over C_\sigma} R(t))$), under the assumption $$0<c_{\alpha_0}=\sup\limits_{t>0}c(\alpha_0,t)=\sup\limits_{t>0}\int_0^t g^{1-\alpha_0}(t-s)f^2(t,s)ds<\infty.$$

Recalling that $(g\circ A^{1\over 2}u)(t)=\int_0^tg(t-s)f^2(t,s)ds$ and $c(\alpha_0,t)<\infty$, also noting ${H}_{1,\alpha_0}(s)=\alpha_0 s^{1-{1\over\alpha_0}}H(s^{1\over\alpha_0})$ is convex for $s$ small, pick $\vartheta$ sufficiently small, by Jensen's Inequality (see Proposition \ref{jensen}) we have
$${H}_{1,\alpha_0}\Big[\vartheta\int_0^tg(t-s)f^2(t,s)ds\Big]={H}_{1,\alpha_0}\Big[\int_0^t\vartheta g^{\alpha_0}(t-s)g^{1-\alpha_0}(s)f^2(t,s)ds\Big]$$
$$= {H}_{1,\alpha_0}\Big[{1\over c(t,\alpha_0)}\int_0^t\vartheta c(t,\alpha_0)g^{\alpha_0}(t-s)g^{1-\alpha_0}(t-s)f^2(t,s)ds\Big]$$
$$\leq {1\over c(t,\alpha_0)}\int_0^t{H}_{1,\alpha_0}\big[\vartheta c(t,\alpha_0)g^{\alpha_0}(t-s)\big] g^{1-\alpha_0}(t-s)f^2(t,s)ds$$
$$={1\over c(t,\alpha_0)}\int_0^t\alpha_0[\vartheta c(t,\alpha_0)g^{\alpha_0}(t-s)\big]^{1-{1\over\alpha_0}} H(\big[\vartheta c(t,\alpha_0)g^{\alpha_0}(t-s)\big]^{1\over\alpha_0})g^{1-\alpha_0}(t-s)f^2(t,s)ds$$
$$={\alpha_0 \vartheta^{1-{1\over\alpha_0}}\over {c(t,\alpha_0)^{1\over\alpha_0}}}\int_0^tH(\big[\vartheta^{1\over\alpha_0} c(t,\alpha_0)^{1\over\alpha_0}g(t-s)\big])f^2(t,s)ds$$

Now we can make $\vartheta$ so small such that $\vartheta^{1\over\alpha_0} c(t,\alpha_0)^{1\over\alpha_0}\leq 1$, thus we have
$$H(\big[\vartheta^{1\over\alpha_0} c(t,\alpha_0)^{1\over\alpha_0}g(t-s)\big])\leq \vartheta^{1\over\alpha_0} c(t,\alpha_0)^{1\over\alpha_0} H(g(t-s)), $$ because $H(0)=0$ and $H(x)$ is convex.
So we have
$${H}_{1,\alpha_0} \Big[\vartheta\int_0^tg(t-s)f^2(t,s)ds\Big]\leq \alpha_0 \vartheta\int_0^t H\big[g(t-s)\big]f^2(t,s)ds$$
$$\leq \alpha_0 \vartheta\int_0^t\big[-g'(t-s)\big]f^2(t,s)ds\leq {\alpha_0\vartheta\over C_\sigma} R(t).$$ The final inequality s from (\ref{c_sigma}).
This completes the proof of (\ref{iteration2}).
\end{proof}

\begin{lemma}\label{est3}
 Given  the results of lemmas above, our energy functional $E(t)$ satisfies
\begin{eqnarray}\label{Main Ineq}
\quad
E(\Tnn)+\hat{H}_{1,\alpha_0}\{E(\Tnn)\}\leq E(\Tn)
\end{eqnarray}
holds for all $n>0$, where $\hat{H}_{1,\alpha_0}$ is defined by $$\hat{H}_{1,\alpha_0}^{-1}(x)={C_1\over\vartheta} T_0 {H}_{1,\alpha_0}^{-1}\Big[{\alpha_0\vartheta\over C_\sigma T_0}x \Big]+ C_2 x, \forall x\in {\bf R}.$$ Moreover, we can show that 
$\hat{H}_{1,\alpha_0}(s)$ is a convex, continuous increasing and zero at the origin function. Here $C_1, C_2, \vartheta$ depends only on $\alpha_0, T_0$, but not on $n$. 
\end{lemma}

\begin{proof}

From (\ref{7th_ineq}), we have
$$E(\Tnn)  \leq  {C_1}  \int_{\Tn}^{\Tnn} (g\circ A^{\frac 12}u)(t)dt + C_2 \int_{\Tn}^{\Tnn} R(t) dt$$
$$ \leq  {C_1\over\vartheta}  \int_{\Tn}^{\Tnn} {H}_{1,\alpha_0}^{-1}({\alpha_0\vartheta\over C_\sigma} R(t))dt + C_2 \int_{\Tn}^{\Tnn} R(t) dt $$
$$ \leq  {C_1\over\vartheta} T_0 {H}_{1,\alpha_0}^{-1}\Big[{\alpha_0\vartheta\over C_\sigma T}\int_{\Tn}^{\Tnn}R(t)dt \Big]+ C_2 \int_{\Tn}^{\Tnn} R(t) dt $$
$$ \triangleq G\Big[\int_{\Tn}^{\Tnn}R(t)dt \Big] .$$
Here we define $G$ by
$$G(x)={C_1\over\vartheta} T_0 {H}_{1,\alpha_0}^{-1}\Big[{\alpha_0\vartheta\over C_\sigma T}x \Big]+ C_2 x.$$
 
It is easy to see that $G$ is increasing and concave near 0 with $G(0)=0$. Denote $G^{-1}$ by $\hat{H}_{1,\alpha_0}$, which is increasing, convex and through origin. Furthermore, 
$$E(\Tnn)   \leq  G\Big[\int_{\Tn}^{\Tnn}R(t)dt \Big] .$$
$$\Rightarrow \hat{H}_{1,\alpha_0}[E(\Tnn)]   \leq  \Big[\int_{\Tn}^{\Tnn}R(t)dt \Big] .$$
$$\Rightarrow \hat{H}_{1,\alpha_0}[E(\Tnn)]   \leq   E(\Tn)-E(\Tnn).$$
$$\Rightarrow E(\Tnn)+\hat{H}_{1,\alpha_0}[E(\Tnn)]   \leq   E(\Tn).$$
\end{proof}

{\bf Step 3: Initial estimate}\\

Define $\underline{H}_{\alpha_0}=I-(I+\hat{H}_{1,\alpha_0})^{-1}$. By Lemma \ref{lemmaLT93} 
and Lemma \ref{est3} we immediately have
\begin{lemma}\label{est4}
Given the lemma \ref{est3}   and condition  (\ref{H_T}) we have  the following decay rates
for the energy function
$$ E(t) \leq s(t), \forall t > T$$ with
\begin{equation}
\label{ODE-1} 
s_t + \underline{H}_{\alpha_0}(s) =0, s(0)=E(0).
\end{equation}
\end{lemma}

Moreover, we have 
\begin{lemma}\label{est40}
All three convex functions, $\underline{H}_{\alpha_0}(x)$, $\hat{H}_{1,\alpha_0}(x)$ and $H_{1,\alpha_0}(x)$, have the same end behavior at origin.
\end{lemma}

\begin{proof}

Note the following
$$\underline{H}_{\alpha_0}=I-(I+\hat{H}_{1, \alpha_0})^{-1}=(I+\hat{H}_{1,\alpha_0})\circ(I+\hat{H}_{1,\alpha_0})^{-1}-(I+\hat{H}_{1,\alpha_0})^{-1}$$
$$=\hat{H}_{1,\alpha_0}\circ(I+\hat{H}_{1,\alpha_0})^{-1}.$$ Since $\hat{H}_{1,\alpha_0}\in C^1[0,\infty)$, we have $\hat{H}_{1,\alpha_0}(x)=O(x)$ at the origin. We use $A \approx B$ to represent that $A$ and $B$ have the same end behaviors. Then 
$$(I+\hat{H}_{1,\alpha_0})\approx I\Rightarrow \hat{H}_{\alpha_0}\circ(I+\hat{H}_{\alpha_0})^{-1}\approx \hat{H}_{1,\alpha_0}\circ I\Rightarrow \underline{H}_{\alpha_0}\approx \hat{H}_{1,\alpha_0},  x\rightarrow 0.$$ For a detailed discussion,  see \cite[Corollary 1. p. 1770]{LT06}.

By similar arguments,  from the relation between $\hat{H}_{1,\alpha_0}(x)$ and $H_{1,\alpha_0}(x)$
$$\hat{H}_{1,\alpha_0}^{-1}(x)={C_1\over\vartheta} T_0 {H}_{1,\alpha_0}^{-1}\Big[{\alpha\vartheta\over C_\sigma T_0}x \Big]+ C_2 x, $$ and the fact that they are both convex at zero,
we can see
$\hat{H}_{1,\alpha_0}(x)\approx H_{1,\alpha_0}(x)$ for $x>0$ in the neighborhood  of $0$. This complete Lemma \ref{est40}.
\end{proof}

\begin{corollary}\label{est40_coro} 
There is $\delta_0>0, \beta >0$ such that when $0\leq x\leq \delta_0$,
$$\underline{H}_{\alpha_0}(x)\geq  \beta H_{1,\alpha_0}(x).$$
\end{corollary}

\subsection{Improve the decay rate}
From Lemmas \ref{est4} and \ref{est40} we know how the natural energy decays: it is bounded by a decay function $s(t)$, which is the solution of an ODE. The ODE portraits the decay property of its solution through the convex function $\hat{H}_{1,\alpha_0}(x)$, which has the same end behavior as $H_{1,\alpha_0}(x)$ near origin. 
So the decay rate is characterized by $H_{1,\alpha_0}(x)$. While this information is quite useful, it is not optimal. We expect the decay rate to be characterized by $H(x)$, which has a better decay than $H_{1,\alpha_0}(x)$ unless $\alpha_0=0$.

Our next step is to improve the decay rates through an iteration process. For that reason, we need a comparison lemma.

\begin{lemma}[Comparison lemma]\label{est5}
Given $y(t)$ satisfying 
\begin{equation}\label{comp} 
y'(t) + H(y(t)) =0, y(0) = y_0 >0, 
\end{equation}
and $s(t)$ satisfying 
\begin{equation} \label{comps} 
s'(t) + \underline{H}_{\alpha_0}(s(t)) =0, s(0) =s_0 > 0, \alpha_0\in (0,1] .
\end{equation}
Both function $y(t)$ and $s(t)$ are positive, decreasing.
Then we have 
\begin{equation}\label{comp3}
s(t)\leq y^{\alpha_0}(\beta t+\kappa),
\end{equation} for some constant $\beta, \kappa$, meaning the decay rate of $s(t)$ is identical to $y^{\alpha_0}(t)$ up to an affine transformation in the coordinates system, which does not change the long time behavior.
\end{lemma}
\begin{proof}
Let $\mathcal{H}(y)=\int_y^\infty\frac {dx}{H(x)}, y>0$. It is easy to see $\mathcal{H}$ is positive and decreasing function on $[0,\infty)$ with $d\mathcal{H}(y)=-{dy\over H(y)}$.  Thus
$$y_t+H(y)=0\Rightarrow -\frac {dy}{H(y)}=dt$$
$$\Rightarrow d\mathcal{H}(y)=dt$$
$$\Rightarrow \mathcal{H}(y)-\mathcal{H}(y_0)=t,$$
$$\Rightarrow y(t)=\mathcal{H}^{-1}(\mathcal{H}(y_0)+t).$$
Similarly, 
$$s_t +\underline{H}_{\alpha_0}(s) = 0$$
$$\Rightarrow s_t\leq -\beta H_{1,\alpha_0}(s)$$
$$\Rightarrow (s^{1\over\alpha_0})_t\leq -\beta H(s^{1\over\alpha_0})$$
$$\Rightarrow {(s^{1\over\alpha_0})_t\over -H(s^{1\over\alpha_0})}\geq \beta $$
$$\Rightarrow {d\over dt} \mathcal{H}(s^{1\over\alpha_0})\geq \beta $$
$$\Rightarrow  \mathcal{H}(s^{1\over\alpha_0})-\mathcal{H}(s_0^{1\over\alpha_0})\geq \beta t$$
$$\Rightarrow s^{1\over\alpha_0}(t)\leq \mathcal{H}^{-1}(\mathcal{H}(s_0^{1\over\alpha_0})+\beta t)=y(\beta t+\kappa)$$
$$\Rightarrow s(t)\leq y^{\alpha_0}(\beta t+\kappa), $$  with  $ \kappa=\mathcal{H}(s_0^{1\over\alpha_0})-\mathcal{H}(y_0).$
\end{proof}
Apply Lemma \ref{comp} with $\alpha_0=1$, we have
\begin{corollary}\label{comp_coro} 
In view of the assumptions on $g(t)$ and $y(t)$, we have $g(t)\leq y(t)$.
\end{corollary}

The previous lemmas show that the energy is bounded by a function $s(t)$, which in turn is bounded by function $y^{\alpha_0}(t)$ as $t\rightarrow\infty$. It might contain a delay, but it makes no difference since we are considering the asymptotic behavior. 

To conclude the proof of  {\bf Theorem \ref{main_theorem}}, we adopt an iteration process.

\begin{lemma}[Iteration for optimality]\label{est6}
In finite steps, we are able to get the following decay rates for the energy function:
There exists positive constants $\tilde{C}, \tilde{\beta}, \tilde{\kappa}$ such that $$E(t)\leq \tilde{C}y(\tilde{\beta} t+\tilde{\kappa}).$$

\end{lemma}

\begin{proof}

By Assumption \ref{main_assumption} and Corollary \ref{comp_coro}, we have $$c(\alpha_0,t)=\int_0^t g^{1-\alpha_0}(t-s)f^2(t,s)ds$$
$$\leq 2E(0)\int_0^t g^{1-\alpha_0}(t-s)ds= 2E(0)\int_0^t g^{1-\alpha_0}(s)ds$$
$$\leq 2E(0)\int_0^\infty y^{1-\alpha_0}(t)dt<\infty.$$ This is the critical estimate to initialize the second part of {\bf Lemma \ref{est2}}. Then by Lemmas \ref{est3}-\ref{est5}, we get

\begin{equation}\label{first_est}
E(t)\leq s(t)=C_1 y^{\alpha_0}(\beta_1 t+\kappa_1 ).
\end{equation}
So we have the first decay estimate. To go to the second iteration and apply Lemma \ref{est2} with $H_{1,2\alpha_0}$, we need
$$c(2\alpha_0,t)=\int_0^tg^{1-2\alpha_0}(t-s)f^2(t,s)ds<\infty.$$

This is true since, by (\ref{first_est})
$$c(2\alpha_0,t)=\int_0^tg^{1-2\alpha_0}(t-s)f^2(t,s)ds$$
$$\leq 2\int_0^tg^{1-2\alpha_0}(t-s)E(s)ds$$
$$\leq 2C_0\int_0^tg^{1-2\alpha_0}(t-s)y^{\alpha_0}(\beta_1 s+\kappa_1)ds<\infty.$$
The last inequality is from Proposition \ref{IrenaSequence} in the Appendix. So we can apply Lemmas \ref{est2} -\ref{est6} again to get $E(t)\leq C_2y^{2\alpha_0}(\beta_2t+\kappa_2)$.

We can continue this iteration with $\alpha_k=k \alpha_0$, until $k$ fist reaches a number $m$ such that 
$\alpha_{m}<1\leq\alpha_{m+1}$. This means 
$$E(t)\leq C_my^{m\alpha_0}(\beta_mt+\kappa_m),$$
hence 
$$\sup\limits_{t>0}\int_0^tE(s)ds\leq C_m\sup\limits_{t>0}\int_0^ty^{m\alpha_0}(\beta_m s+\kappa_m)ds<\infty,$$
which is guaranteed by Proposition \ref{IrenaSequence}.
It follows that $$c(1,t)=\sup\limits_{t>0}\int_0^tf^2(t,s)ds\leq\sup\limits_{t>0}\int_0^t2E(s)ds<\infty,$$
in the next iteration when we apply Lemma \ref{est2}, we simply pick $H_{1,1}(s)=H(s)$ to start the final iteration.
And we end up with $E(t)\leq \tilde{C}y(\tilde{\beta} t+\tilde{\kappa})$.

\end{proof}

\section{Appendix}

\begin{proposition}[Jensen's inequality	]\label{jensen}
Let $\Omega$ be a measurable subset of the real line. Let $f(x)$ be a non-negative function on $\Omega$ with $k=\int_\Omega f(x)ds$ finite. If $g(x)$ is a real-valued measurable function on $\Omega$ and function $\varphi$ is convex on the range of $g$. Then we have
$$\varphi\left({1\over k}\int_\Omega g(x)f(x)dx\right)\leq {1\over k}\int_\Omega \varphi(g(x))f(x)dx.$$
 \end{proposition}

\begin{remark}
Search Jensen's inequality online for the proof.
\end{remark}

\begin{proposition}[$\alpha$-Sequence]\label{IrenaSequence}
Let $\alpha_0\in (0,1)$ and $y(t)\in C[0,\infty)$ be a positive function decreasing to zero. Moreover $$\int_0^\infty y^{1-\alpha_0}(t)dt=L<\infty.$$
Let $m$ be a positive integer such that $m\alpha_0<1$ and $(m+1)\alpha_0\geq 1$. Let $\beta , \kappa$ be positive numbers with $\beta >1$. A finite sequence of functions, in form of definite integrals on $[0, t]$, are generated in the following way:
\begin{eqnarray}\label{ineq32}
I_k(t)&=&\int_0^t y^{1-k\alpha_0}(t-s)y^{(k-1)\alpha_0}(\beta s+\kappa)ds, k=1, 2, ..., m,\\
I_{m+1}(t)&=&\int_0^t y^{m\alpha_0}(\beta s+\kappa)ds.
\end{eqnarray}

Then each $I_k(t), k=1,..., m+1$, is bounded, uniformly in $t$, namely, $$\sup\limits_{t>0}I_k(t)<\infty.$$
\end{proposition}

\begin{proof}

We first assume $y(0)\leq 1$.

For $k=1, ..., m$, we have
$$I_k=\int_0^ty^{1-k\alpha_0}(t-s)y^{(k-1)\alpha_0}(\beta s+\kappa)ds$$
$$=\int_0^{t/2}y^{1-k\alpha_0}(t-s)y^{(k-1)\alpha_0}(\beta s+\kappa)ds+\int_{t/2}^ty^{1-k\alpha_0}(t-s)y^{(k-1)\alpha_0}(\beta s+\kappa)ds$$
$$\leq \int_0^{t/2}y^{1-k\alpha_0}(s)y^{(k-1)\alpha_0}(\beta s+\kappa)ds+\int_{t/2}^ty^{1-k\alpha_0}(t-s)y^{(k-1)\alpha_0}(\beta (t-s)+\kappa)ds$$
$$= \int_0^{t/2}y^{1-\alpha_0}(s)ds+\int_{t/2}^ty^{1-\alpha_0}(t-s)ds$$
$$= 2\int_0^{t/2}y^{1-\alpha_0}(s)ds\leq 2L<\infty.$$
Above we use the fact that $y(t)$ is a decreasing function, hence $y(\beta s+\kappa)\leq y(s)$.

For $I_{m+1}$, we have
$$I_{m+1}=\int_0^t y^{m\alpha_0}(\beta s+\kappa)ds
\leq\int_0^t y^{1-\alpha_0}(\beta s+\kappa)ds\leq L<\infty,$$ since $(m+1)\alpha_0>1\Rightarrow m\alpha_0>1-\alpha_0$ and $y(t)\leq y(0)\leq 1$.

For $y$ with $y(0)>1$, we can always find a time $t_0$ so that $y(t)\leq 1$ on $[t_0,\infty)$, since $y$ decreases to zero. Whether or not the integrals are finite only depends upon the asymptotic behavior of $y$ at infinity. So $y$ being larger than $1$ on a  finite interval $[0, t_0]$ does not bear influence on our result.
\end{proof}

The research of I. Lasiecka and X. Wang has been  partially supported by NSF  Grant: DMS 0104305. 


\end{document}